\documentclass[a4paper,12pt]{article}
\usepackage{amsmath,amssymb,amsthm}
\usepackage{amscd}

\makeatletter

\@addtoreset{equation}{subsection}
\makeatother

\newtheorem{dfn}{Definition}[subsection]
\newtheorem{prop}[dfn]{Proposition}
\newtheorem{thm}[dfn]{Theorem}
\newtheorem{lem}[dfn]{Lemma}
\newtheorem{cor}[dfn]{Corollary}

\newtheorem{problem}[dfn]{Problem}

\newcommand\ad{{\rm ad}}

\newcommand\gfg{{\mathfrak{g}}}

\newcommand\Ker{{\rm Ker}}
\newcommand\pa{{\partial}}

\begin{document}

\title
{Surface topology and involutive bimodules}
\author{Nariya Kawazumi}
\maketitle

\begin{abstract} We remark some basic facts on 
homological aspects of involutive Lie bialgebras and 
their involutive bimodules, and present some 
problems on surface topology related to these facts. 
\end{abstract}

\maketitle

\begin{center}
Introduction
\end{center}

The notion of a Lie bialgebra was originated by Drinfel'd 
in the celebrated paper \cite{D}. There he observed that any 
bialgebra structure on a fixed Lie algebra $\gfg$ is regarded 
as a $1$-cocycle of $\gfg$ with values in the second exterior power 
$\Lambda^2\gfg$, and that the coboundary of any element
in $\Lambda^2\gfg$ satisfying the Yang-Baxter equation
defines a Lie bialgebra 
structure on the Lie algebra $\gfg$. It can be regarded as 
a deformation of the Lie bialgebra structure on $\gfg$ with 
the trivial coalgebra structure. \par
It was Turaev \cite{T2} who discovered a close relation between 
surface topology and the notion of a Lie bialgebra. Let $S$ be a 
connected oriented surface, and $\mathbb{Q}\hat\pi(S)$ 
the (rational) Goldman Lie algebra of the surface $S$
\cite{Go}, which is the $\mathbb{Q}$-free vector space over the homotopy set 
$\hat\pi(S) = [S^1, S]$ of free loops on the surface $S$ equipped 
with the Goldman bracket. 
The constant loop $1$ is in the center of $\mathbb{Q}\hat\pi(S)$, so that 
the quotient $\mathbb{Q}\hat\pi'(S) :=  \mathbb{Q}\hat\pi(S)/\mathbb{Q}1$ 
has a natural Lie algebra structure. He introduced a natural cobracket, 
the Turaev cobracket, on $\mathbb{Q}\hat\pi'(S)$, and proved that 
it is a Lie bialgebra. Later Chas \cite{Chas} proved that it 
satisfies the involutivity. See Appendix for the definition of these operations. \par
On the other hand, Schedler \cite{Sch} introduced a natural involutive Lie bialgebra structure on the necklace Lie algebra associated to a quiver.  
Let $H$ be a symplectic $\mathbb{Q}$-vector space of dimension $2g$, 
$g \geq 1$, and $\widehat{T} := \prod^\infty_{m=0}H^{\otimes m}$ 
the completed tensor algebra over $H$. 
We denote by $\mathfrak{a}_g^- = {\rm Der}_\omega(\widehat{T})$ 
the Lie algebra of continuous derivations 
on $\widehat{T}$ annihilating the symplectic form $\omega \in H^{\otimes 2}$. 
It includes Kontsevich's ``associative" $a_g$ as a Lie subalgebra. 
The Lie algebra $\mathfrak{a}_g^-$ is the necklace Lie algebra 
associated to some quiver. Hence it is an involutive Lie bialgebra by Schedler's 
cobracket. Massuyeau \cite{Mas} introduced the notion of a symplectic 
expansion of 
the fundamental group of $\Sigma_{g,1}$, a compact connected oriented 
surface of genus $g$ with $1$ boundary component. 
Kuno and the author \cite{KK1} \cite{KK3} proved that 
a natural completion of the Lie algebra $\mathbb{Q}\hat\pi'(\Sigma_{g,1})$ is 
isomorphic to 
the Lie algebra $\mathfrak{a}_g^-$ by using a symplectic expansion. 
In particular, the Turaev cobracket defines an involutive 
Lie bialgebra structure on the Lie algebra $\mathfrak{a}_g^-$, which depends 
on the choice of a symplectic expansion, and does not coincide with 
Schedler's cobracket. In \S4 we present some problems related to these 
cobrackets. 
\par
Now we go back to an arbitrary connected oriented surface $S$. 
Suppose that its boundary $\partial S$ is non-empty.  
Then choose two (not necessarily distinct) points $*_0$ and $*_1$ 
in $\partial S$. 
We denote by $\Pi S(*_0, *_1)$ the homotopy set 
of paths from $*_0$ to $*_1$, namely $[([0,1], 0, 1), (S, *_0, *_1)]$. In \cite{KK1} and \cite{KK3} Kuno and 
the author discovered that $\mathbb{Q}\Pi S(*_0, *_1)$, the $\mathbb{Q}$-free 
vector space over the set $\Pi S(*_0, *_1)$, 
is a nontrivial $\mathbb{Q}\hat\pi'(S)$-module in a natural way. 
Moreover, inspired by \cite{T1}, they \cite{KK4} introduced a natural operation 
$$
\mu: \mathbb{Q}\Pi S(*_0, *_1) \to \mathbb{Q}\Pi S(*_0, *_1)\otimes \mathbb{Q}\hat\pi'(S).
$$
It should satisfy some natural properties analogous to the defining conditions 
of an involutive Lie bialgebra. So, in \cite{KK4}, 
they introduced the defining conditions 
of an involutive 
$\mathbb{Q}\hat\pi'(S)$-module, and proved that $\mu$ satisfies all the 
 conditions. See also Appendix for details.
As applications of the compatibility condition among them, 
they \cite{KK4} obtain a 
criterion for the non-realizability of generalized Dehn twists \cite{Ku2}, 
and a geometric constraint of the (geometric) Johnson homomorphism 
of the (smallest) Torelli group. \par
The purpose of the present paper is to explain a homological background 
of the defining conditions of an involutive Lie bialgebra and 
its involutive bimodule, 
and to present some problems on surface topology related to this background. 
Our key observation is the classical fact: {\it the Jacobi identity for 
a Lie algebra $\gfg$ is equivalent to the integrability condition 
$\pa\pa = 0$ on the exterior algebra $\Lambda^*\gfg$.} 
Throughout this paper we work over the rationals $\mathbb{Q}$
for simplicity. But all the propositions in this paper hold good over any 
field of characteristic $0$. Let $\gfg$ be a Lie algebra over $\mathbb{Q}$, 
$\pa: \Lambda^p\gfg \to \Lambda^{p-1}\gfg$, $p \geq 1$, the standard 
boundary operator. See, for example, \cite{CE}.  
Moreover let $\delta: \gfg \to \Lambda^2\gfg$ be 
a $\mathbb{Q}$-linear map.  
The map $\delta$ has a natural extension 
$d: \Lambda^p\gfg \to \Lambda^{p+1}\gfg$ for any $p \geq 0$. 
Then we have
\begin{prop}\label{0bialg}
The pair $(\gfg, \delta)$ is an involutive Lie bialgebra, 
if and only if $dd = 0$ and $d\pa + \pa d = 0$ on $\Lambda^*\gfg$. 
\end{prop}
This is an easy exercise. But, to complete our argument, we prove it 
in \S1. The proposition implies the homology group $H_*(\gfg)$ 
of the Lie algebra $\gfg$ is a cochain complex with the coboundary 
operator $d(\delta) := H_*(d)$, if $\gfg$ is an involutive Lie bialgebra. 
\begin{problem}
Find a meaning 
of the cohomology group $H^*(H_*(\gfg), d(\delta))$ for any involutive Lie bialgebra $(\gfg, \delta)$. 
\end{problem}
\par
Suppose $\gfg$ is an involutive Lie bialgebra. Let $M$ be a $\gfg$-module. 
Then we can consider the standard chain complex $(M\otimes\Lambda^*\gfg, 
\pa)$ of the Lie algebra $\gfg$ with values in $M$ \cite{CE}. 
Any $\mathbb{Q}$-linear map $\mu: M \to M\otimes \gfg$ has a natural 
extension $d = d^M: M\otimes \Lambda^p\gfg \to M\otimes \Lambda^{p+1}\gfg$ 
for any $p \geq 0$. Then we have 
\begin{prop}\label{0bimod}
 The pair $(M, \mu)$ is an involutive $\gfg$-bimodule in the sense of 
\cite{KK4}, if and only if $dd = 0$ and $d\pa + \pa d = 0$ on $M\otimes\Lambda^*\gfg$. 
\end{prop}
Similarly to $H_*(\gfg)$, the homology group $H_*(\gfg; M)$ of $\gfg$ with values in $M$ admits the coboundary operator $d(\delta, \mu) := H_*(d)$
 if $M$ is an involutive $\gfg$-bimodule. 
\begin{problem} Let $(\gfg, \delta)$ be an involutive Lie bialgebra. Then 
find a meaning 
of the cohomology group $H^*(H_*(\gfg; M), d(\delta, \mu))$
for any involutive $\gfg$-bimodule $(M, \mu)$. 
\end{problem}
\par
In \S3 we study Drinfel'd's deformation of a Lie bialgebra structure 
by a $1$-coboundary stated above. 
We can consider an analogous deformation of 
an involutive bimodule. We prove that such a deformation does not 
affect the coboundary operators $d(\delta)$ and $d(\delta, \mu)$ 
on $H_*(\gfg)$ and $H_*(\gfg; M)$
(Lemma \ref{A4coboundary} and Proposition \ref{A8Aexp}). 
In \S4 we discuss some relation among these homological facts and surface topology, in particular, a tensorial description of the Turaev cobracket and Kontsevich's non-commutative symplectic geometry. 
In Appendix we briefly review some operations of loops on a surface 
\cite{Go} \cite{T2} \cite{KK1} \cite{KK4}. 
\par
We conclude the introduction by listing our convention of notation in this paper. 
For a $\mathbb{Q}$-vector space $V$ and $p \geq 1$, the $p$-th 
symmetric group $\mathfrak{S}_p$ acts on the tensor space 
$V^{\otimes p}$ by permuting the components. 
In particular, we denote $T:= (12) \in \operatorname{Aut}(V^{\otimes 2})$
and $N := 1 + (123) + (123)^2 \in \operatorname{End}(V^{\otimes 3})$. 
We regard the $p$-th exterior power $\Lambda^pV$ as a linear subspace 
of $V^{\otimes p}$ in an obvious way
$
\Lambda^pV := \{u \in V^{\otimes p}; \sigma(u) =
(\operatorname{sgn}\sigma)u\}.
$
For $X_i \in V$, $1 \leq i \leq p$, we identify 
$
X_1\wedge\cdots\wedge X_p = \sum_{\sigma\in \mathfrak{S}_p}
(\operatorname{sgn}\sigma)X_{\sigma(1)}\cdots X_{\sigma(p)} 
\in \Lambda^pV \subset V^{\otimes p}.
$
Here and throughout this paper we omit the symbol $\otimes$,
if there is no fear of confusion.
In particular, we have 
\begin{equation}
\wedge = (1-T): V^{\otimes 2} \to \Lambda^2V, \quad
XY \mapsto X\wedge Y = (1-T)(XY).
\label{A0wedge}
\end{equation}

\medskip
\noindent \textbf{Acknowledgments.}
First of all, the author thanks Yusuke Kuno for lots of valuable discussions
and his comments for the first draft of this paper. 
This paper is a byproduct of our joint paper \cite{KK4}. 
He also thanks Atsushi Matsuo, Robert Penner and especially Gwenael 
Massuyeau for lots of helpful conversations.
\par

\tableofcontents

\section{Lie bialgebras} 

In this section we recall the definitions of a Lie algebra, a Lie coalgebra and a Lie bialgebra, and prove Proposition \ref{0bialg}. 

\subsection{Lie algebras}\label{LA}

Let $\gfg$ be a $\mathbb{Q}$-vector space equipped with a $\mathbb{Q}$-linear map
$\nabla: \gfg\otimes\gfg \to \gfg$ satisfying {\bf the skew condition} 
\begin{equation}
\nabla T = -\nabla: \gfg^{\otimes 2} \to \gfg.
\label{A1skew}
\end{equation}
Following the ordinary terminology, we denote $[X, Y] := \nabla(X\otimes Y)$ for any $X$ and $Y \in \gfg$. 
Then we define $\mathbb{Q}$-linear maps 
$\sigma: \gfg\otimes\Lambda^p\gfg \to \Lambda^p\gfg$ and 
$\pa: \Lambda^{p}\gfg \to \Lambda^{p-1}\gfg$ by 
\begin{eqnarray*}
\sigma(Y)(X_1\wedge\cdots\wedge X_p) 
&:=& \sum^p_{i=1} X_1\wedge\cdots\wedge X_{i-1}\wedge [Y, X_i]\wedge
X_{i+1}\wedge\cdots\wedge X_p,\\
\pa(X_1\wedge\cdots\wedge X_p)
&:=& \sum_{i<j}(-1)^{i+j} [X_i, X_j]\wedge X_1\wedge \overset{\hat
i}\cdots \overset{\hat j}\cdots\wedge X_p,
\end{eqnarray*}
for $X_i$ and $Y \in \gfg$. 
It is easy to show
\begin{equation}
\pa(X_1\wedge\cdots\wedge X_p\wedge Y) = 
\pa(X_1\wedge\cdots\wedge X_p)\wedge Y 
+ (-1)^{p+1}\sigma(Y)(X_1\wedge\cdots\wedge X_p).
\label{A1partial}
\end{equation}

\begin{lem}\label{A1dd=0}
We have $\pa\pa=0: \Lambda^*\gfg \to \Lambda^*\gfg$, if and only if
$\nabla$ satisfies {\bf the Jacobi identity}
\begin{equation}
\nabla(\nabla\otimes 1)N = 0: \gfg^{\otimes 3} \to \gfg.
\label{A1Jacobi}
\end{equation}
\end{lem}

\begin{proof}
For $X$, $Y$ and $Z \in \gfg$, we have 
$$
\pa\pa(X\wedge Y\wedge Z) = 
[[X, Y], Z] + [[Y, Z], X] + [[Z, X], Y]. 
$$
Hence $\pa\pa=0$ implies the Jacobi identity. \par
Assume the Jacobi identity. Then, by some straight-forward computation, 
we have 
\begin{equation}
\sigma(Y)\pa(X_1\wedge\cdots\wedge X_p) = 
\pa\sigma(Y)(X_1\wedge\cdots\wedge X_p)
\label{A1action}
\end{equation}
for any $X_i$ and $Y \in \gfg$. This proves $\pa\pa = 0: 
\Lambda^p\gfg \to \Lambda^{p-2}\gfg$ by induction on $p \geq 2$. 
In the case $p=2$, $\pa\pa = 0$ is trivial.
Assume $\pa\pa = 0: \Lambda^p\gfg \to \Lambda^{p-2}\gfg$ for $p \geq 2$. 
Then, using (\ref{A1partial}) and (\ref{A1action})
for $\xi \in \Lambda^p\gfg$ and $Y \in\gfg$, we compute
\begin{eqnarray*}
&&\pa\pa(\xi\wedge Y)
= \pa((\pa\xi)\wedge Y + (-1)^{p+1}\sigma(Y)\xi)\\
&=& (\pa\pa\xi)\wedge Y + (-1)^p\sigma(Y)\pa\xi
+ (-1)^{p+1}\pa(\sigma(Y)\xi)
= (\pa\pa\xi)\wedge Y = 0
\end{eqnarray*}
by the inductive assumption. This proves the lemma. 
\end{proof}

The pair $(\gfg, \nabla)$ is called a {\bf Lie algebra} if the map 
$\nabla$ satisfies the Jacobi identity (\ref{A1Jacobi}). 
The map $\nabla$ is called the {\bf bracket} of the Lie algebra. 
Then the $p$-th homology group of the chain complex 
$\Lambda^*\gfg = \{\Lambda^p\gfg, \pa\}_{p\geq0}$ is denoted by 
$$
H_p(\gfg) = H_p(\Lambda^*\gfg)
$$
and called the $p$-th homology group of the Lie algebra $\gfg$.
See, for example, \cite{CE}. \par

For any Lie algebra $\gfg$, by some straight-forward computation, 
one can prove the following, which will be used in \S\ref{LBA}. 
\begin{lem}\label{A1wedge}
For $\xi = X_1\wedge\cdots\wedge X_p \in \Lambda^p\gfg$ and 
$\eta = Y_1\wedge\cdots\wedge Y_q \in \Lambda^q\gfg$, 
$X_i, Y_j \in \gfg$, 
$$
\pa(\xi\wedge\eta) - (\pa\xi)\wedge\eta -(-1)^p\xi\wedge \pa\eta
= \sum^p_{i=1} (-1)^i X_1\wedge\overset{\hat i}\cdots\wedge X_p\wedge
\sigma(X_i)(\eta).
$$
\end{lem}

\subsection{Lie coalgebras}\label{LCA}

Next we consider a $\mathbb{Q}$-vector space equipped 
with a $\mathbb{Q}$-linear map
$\delta: \gfg\to \gfg\otimes\gfg$ satisfying {\bf the coskew condition}
\begin{equation}
T\delta = -\delta: \gfg\to \gfg^{\otimes 2}.
\label{A2coskew}
\end{equation}
We may regard $\delta(\gfg) \subset \Lambda^2\gfg$.  
Then we define a $\mathbb{Q}$-linear map 
$d: \Lambda^{p}\gfg \to \Lambda^{p+1}\gfg$, $p \geq 0$, by 
$d\vert_{\Lambda^0\gfg} := 0$ and 
$$
d(X_1\wedge\cdots\wedge X_p) := 
\sum^p_{i=1}(-1)^i(\delta X_i)\wedge X_1\wedge\overset{\hat
i}\cdots\wedge X_p
$$
for any $p \geq 1$ and $X_i \in \gfg$. 
In particular, $dX = -\delta X$ for $X \in \gfg$. 
If $\xi \in \Lambda^p\gfg$ and $\eta \in \Lambda^q\gfg$, then
\begin{equation}
d(\xi\wedge\eta) = (d\xi)\wedge \eta + (-1)^p\xi\wedge(d\eta).
\label{A2wedge}
\end{equation}

\begin{lem}\label{A2dd=0}
We have $dd=0: \Lambda^*\gfg \to \Lambda^*\gfg$, if and only if
$\delta$ satisfies {\bf the coJacobi identity}
\begin{equation}
N(\delta\otimes 1)\delta = 0: \gfg \to \gfg^{\otimes 3}.
\label{A2coJacobi}
\end{equation}
\end{lem}

\begin{proof}
If we denote $\delta X = \sum_i X'_i\wedge X''_i$, 
$X'_i, X''_i \in\gfg$, then we have 
\begin{eqnarray*}
&& (\delta X)\wedge Y = \sum X'_i\wedge X''_i\wedge Y\\
&=& \sum X'_iX''_iY + X''_iYX'_i + YX'_iX''_i 
- X''_iX'_iY - X'_iYX''_i - YX''_iX'_i\\
&=& N((\delta X)Y).
\end{eqnarray*}
This implies $d(X\wedge Y) = -(\delta X)\wedge Y + (\delta Y)\wedge X
= -N((\delta X)Y) + N((\delta Y)X)
= -N(\delta\otimes 1)(XY-YX) = -N(\delta\otimes 1)(X\wedge Y)$.
Since $\delta\gfg \subset \Lambda^2\gfg$, we obtain
\begin{equation}
dd = N(\delta\otimes 1)\delta: \gfg \to \gfg^{\otimes 3}.
\label{A2dd}
\end{equation}
Hence $dd=0$ implies the coJacobi identity. \par
Assume the coJacobi identity. 
We prove $dd = 0: 
\Lambda^p\gfg \to \Lambda^{p+2}\gfg$ by induction on $p \geq 1$. 
In the case $p=1$, $dd = 0$ is equivalent to the coJacobi identity.
Assume $dd = 0: \Lambda^p\gfg \to \Lambda^{p+2}\gfg$ for $p \geq 1$. 
Then, for $\xi \in \Lambda^p\gfg$ and $Y \in\gfg$, we have 
$dd(\xi\wedge Y) = d((d\xi)\wedge Y + (-1)^p\xi\wedge dY) 
= (dd\xi)\wedge Y + (-1)^{p+1}(d\xi)\wedge dY + (-1)^pd\xi\wedge dY
+\xi\wedge ddY = (dd\xi)\wedge Y +\xi\wedge ddY = 0$ 
by the inductive assumption. This proves the lemma. 
\end{proof}

The pair $(\gfg, \delta)$ is called a {\bf Lie coalgebra} if the map 
$\delta$ satisfies the coJacobi identity (\ref{A2coJacobi}). 
The map $\delta
$ is called the {\bf cobracket} of the Lie coalgebra. 
Then the $p$-th cohomology group of the cochain complex 
$\Lambda^*\gfg = \{\Lambda^p\gfg, d\}_{p\geq0}$ is denoted by 
$$
H^p(\gfg) = H^p(\Lambda^*\gfg)
$$
and called the $p$-th cohomology group of the Lie coalgebra $\gfg$. 
In view of the formula (\ref{A2wedge}), $H^*(\gfg)$ is a graded 
commutative algebra. 
\par
Assume $\gfg$ is a complete filtered $\mathbb{Q}$-vector space, 
i.e., there exists a decreasing filtration $\gfg = F_0\gfg \supset
F_1\gfg \supset \cdots \supset F_n\gfg \supset F_{n+1}\gfg \supset\cdots$ 
such that the completion map $\gfg \to \widehat{\gfg} 
:= \varprojlim_{n\to\infty}\gfg/F_n\gfg$ is an isomorphism. 
Then we can consider a $\mathbb{Q}$-linear map 
$\delta: \gfg \to \gfg\widehat{\otimes}\gfg$, whose target is 
the completed tensor product of two copies of $\gfg$. 
Then the pair $(\gfg, \delta)$ is a complete Lie coalgebra 
if the map $\delta$ satisfies the coskew condition (\ref{A2coskew})
and the coJacobi identity (\ref{A2coJacobi}), where $\gfg^{\otimes 2}$ 
and $\gfg^{\otimes 3}$ are replaced by the completed tensor product
$\gfg^{\widehat{\otimes} 2}$ and $\gfg^{\widehat{\otimes} 3}$, 
respectively. In this case we consider the $p$-th complete exterior 
power, i.e., the alternating part of $\gfg^{\widehat{\otimes} p}$, 
instead of $\Lambda^p\gfg$ for any $p \geq 0$. \par

\subsection{Involutive Lie bialgebras}\label{LBA}

Let $(\gfg, \nabla)$ be a Lie algebra, and $(\gfg, \delta)$ a Lie 
coalgebra with the same underlying vector space $\gfg$. 
We look at the operator $d\pa + \pa d: \Lambda^p\gfg \to 
\Lambda^p\gfg$ for $p \geq 0$. It is clear $d\pa + \pa d = 0$ 
for $p=0$. 
\begin{lem}\label{A3axiom}
We have $d\pa + \pa d = 0: \Lambda^p\gfg \to 
\Lambda^p\gfg$ for $p=1$ and $2$, if and only if
$\nabla$ and $\delta$ satisfy {\bf the compatibility condition}
\begin{equation}
\forall X, \forall Y \in \gfg, \quad
\delta [X, Y] = \sigma(X)(\delta Y) - \sigma(Y)(\delta X),
\label{A3compatible}
\end{equation}
and {\bf the involutivity}
\begin{equation}
\nabla\delta = 0: \gfg \to \gfg.
\label{A3involutive}
\end{equation}
\end{lem}
\begin{proof}
From the definition, the involutivity is equivalent to 
$d\pa+\pa d = 0$ for $p=1$. Assume the involutivity. 
Then, for $X$ and $Y \in \gfg$, we have 
$(d\pa+\pa d)(X\wedge Y) 
= -d[X, Y] +\pa((dX)\wedge Y - X\wedge (dY))
= \delta [X, Y] + (\pa dX)\wedge Y - \sigma(Y)(dX)
- (\pa dY)\wedge X - \sigma(X)(dY)
= \delta [X, Y] + \sigma(Y)(\delta X) - \sigma(X)(\delta Y)$.
Hence $d\pa+\pa d = 0$ for $p=2$ is equivalent 
to the compatibility condition. This proves the lemma.
\end{proof}

When the compatibility condition holds, $\gfg$ is called a {\bf 
Lie bialgebra}. This is the definition given by Drinfel'd in \cite{D}. 
A Lie bialgebra $\gfg$ is called 
{\bf involutive}, if it satisfies the involutivity.
\par

\begin{lem}\label{A3partial}
If $\gfg$ is a Lie bialgebra, we have
$$
\pa(\xi\wedge dY) - (\pa\xi)\wedge dY - (-1)^p\xi\wedge \pa dY
= d\sigma(Y)\xi - \sigma(Y)d\xi
$$
for $\xi \in \Lambda^p\gfg$ and $Y \in \gfg$. 
\end{lem}
\begin{proof}
It suffices to show the lemma for $\xi = X_1\wedge\cdots\wedge X_p$, 
$X_i \in \gfg$. By the compatibility condition, we have 
\begin{eqnarray*}
&& d\sigma(Y)\xi - \sigma(Y)d\xi
= \sum^p_{i=1}(-1)^{i-1}X_1\wedge\cdots\wedge
(d[Y, X_i] - \sigma(Y)dX_i)\wedge\cdots\wedge X_p \\
&=& \sum^p_{i=1}(-1)^{i}X_1\wedge\cdots\wedge
\sigma(X_i)dY\wedge\cdots\wedge X_p 
= \sum^p_{i=1}(-1)^{i}X_1\wedge\overset{\hat i}\cdots\wedge X_p\wedge
\sigma(X_i)dY,
\end{eqnarray*}
which equals $\pa(\xi\wedge dY) - (\pa\xi)\wedge dY 
- (-1)^p\xi\wedge \pa dY$ from Lemma \ref{A1wedge}. 
This proves the lemma.
\end{proof}

\begin{prop}\label{A3dd+dd}
If $\gfg$ is a Lie bialgebra, then we have 
$$
(d\pa+\pa d)(X_1\wedge\cdots\wedge X_p) 
= \sum^p_{i=1} X_1\wedge\cdots\wedge X_{i-1}\wedge (\pa dX_i)\wedge
X_{i+1}\wedge\cdots\wedge X_p
$$
for $X_i \in \gfg$. 
\end{prop}
\begin{proof}
It is clear for $p=1$. Assume it holds for $p \geq 1$.
Denote $\xi = X_1\wedge\cdots\wedge X_p$ and $Y = X_{p+1}$. 
Then, from Lemma \ref{A3partial}, 
$(d\pa +\pa d)(\xi\wedge Y) 
= d((\pa\xi)\wedge Y + (-1)^{p+1}\sigma(Y)\xi)
+ \pa((d\xi)\wedge Y + (-1)^p\xi\wedge dY) 
= (d\pa Y)\wedge Y + (-1)^{p+1}(\pa\xi)\wedge dY
+(-1)^{p+1}d\sigma(Y)\xi + (\pa d\xi)\wedge Y
+ (-1)^{p+2}\sigma(Y)d\xi + (-1)^p\pa(\xi\wedge dY)
= ((d\pa + \pa d)\xi)\wedge Y + \xi\wedge\pa dY$.
This proceeds the induction.
\end{proof}

\begin{cor}\label{A3dd=0}
A Lie bialgebra $\gfg$ satisfies $d\pa + \pa d = 0: \Lambda^p\gfg \to
\Lambda^p\gfg$ for any $p \geq 0$, if and only if $\gfg$ is involutive.
\end{cor}
This completes the proof of Proposition \ref{0bialg} stated in Introduction.\par
For an involutive Lie bialgebra $\gfg$, the operator $d$ induces 
the coboundary operator 
\begin{equation}
d = d(\delta): H_p(\gfg) \to H_{p+1}(\gfg), \quad
[u] \mapsto [du]
\label{A3d}\end{equation}
on the homology group $H_*(\gfg)$. 
Hence one can define the cohomology of the homology 
$H^*(H_*(\gfg))$. \par
When the pair $(\gfg, \delta)$ is a complete Lie coalgebra, 
we have to assume that the bracket $\nabla$ is continuous
with respect to the filtration of $\gfg$, and to 
replace the exterior algebra $\Lambda^*\gfg$ 
by the complete exterior algebra of $\gfg$ in the three propositions 
in this subsection.  Then all of them hold good. 
In particular, we can consider a complete Lie bialgebra and 
a complete involutive Lie bialgebra. 
Similarly we can 
consider a complete comodule and a complete (involutive)
bimodule in the next section. \par

\section{Bimodules}

We discuss a homological background of the defining conditions of 
an involutive bimodule introduced by Kuno and the author in \cite{KK4}. 
In other words, we prove Proposition \ref{0bimod} stated in Introduction.

\subsection{Modules}

Let $\gfg$ be a Lie algebra, $M$ a $\mathbb{Q}$-vector space 
equipped with a 
$\mathbb{Q}$-linear map $\sigma: \gfg\otimes M \to M$, $X\otimes m
\mapsto Xm$. We define a $\mathbb{Q}$-linear map $\Gamma_\sigma = \Gamma:
M\otimes\Lambda^p\gfg \to M\otimes\Lambda^{p-1}\gfg$ by 
$\Gamma(m\otimes X_1\wedge\cdots\wedge X_p) :=
\sum^p_{i=1}(-1)^i(X_im)\otimes X_1\wedge\overset{\hat i}\cdots\wedge
X_p$ for $p \geq 1$, $m \in M$ and $X_i \in \gfg$, and a
$\mathbb{Q}$-linear map $\pa^M = \pa: M \otimes\Lambda^p\gfg \to
M\otimes\Lambda^{p-1}\gfg$ by $\pa(m\otimes\xi) := \Gamma(m\otimes\xi) +
m\otimes\pa(\xi)$ for $m \in M$ and $\xi \in \Lambda^p\gfg$. Here $\pa:
\Lambda^p\gfg \to \Lambda^{p-1}\gfg$ is the operator introduced in 
\S\ref{LA}. By some straight-forward computation, we have 
\begin{equation}
\Gamma(m\otimes\xi\wedge\eta) 
= \Gamma(m\otimes\xi)\wedge\eta + (-1)^{pq}\Gamma(m\otimes\eta)\wedge\xi
\label{A5Gpartial}
\end{equation}
for any $m \in M$, $\xi\in \Lambda^p\gfg$ and $\eta \in \Lambda^q\gfg$. 
Furthermore we define a $\mathbb{Q}$-linear map $\sigma:
\gfg\otimes M\otimes\Lambda^p\gfg \to M\otimes\Lambda^p\gfg$ by 
$
\sigma(Y)(m\otimes\xi) := (Ym)\otimes\xi + m\otimes\sigma(Y)(\xi)
$
for $Y \in \gfg$, $m \in M$ and $\xi \in \Lambda^p\gfg$. Then it is easy
to show
\begin{equation}
\pa(m\otimes\xi\wedge Y) = \pa(m\otimes\xi)\wedge Y +
(-1)^{p+1}\sigma(Y)(m\otimes\xi).
\label{A5partial}
\end{equation}

\begin{lem}\label{A5dd=0}
We have $\pa^M\pa^M = 0: M\otimes\Lambda^*\gfg \to
M\otimes\Lambda^*\gfg$, if and only if the condition
\begin{equation}
\forall X, \forall Y \in \gfg, \forall m \in M, \quad 
[X, Y]m = X(Ym) - Y(Xm)
\label{A5action}
\end{equation}
holds. 
\end{lem}
\begin{proof}
For $X, Y \in \gfg$ and $m \in M$, we have
$$
\pa\pa(m\otimes X\wedge y) = [X, Y]m - X(Ym) + Y(Xm).
$$
Hence $\pa^M\pa^M = 0$ implies the condition (\ref{A5action}). \par
Assume the condition (\ref{A5action}).
Then it is easy to show
\begin{equation}
\sigma(Y)\Gamma(m\otimes X_1\wedge\cdots\wedge X_p) 
= \Gamma(\sigma(Y)(m\otimes X_1\wedge\cdots\wedge X_p))
\label{A5equiv}
\end{equation}
for any $m \in M$ and $Y, X_i \in \gfg$.  
From this formula and (\ref{A1action}) follows 
\begin{equation}
\sigma(Y)\pa(m\otimes \xi) = \pa(\sigma(Y)(m\otimes\xi))
\label{A5comm}\end{equation}
for any $m \in X$, $Y \in \gfg$ and $\xi \in \Lambda^p\gfg$. 
This proves $\pa\pa = 0: 
M\otimes \Lambda^p\gfg \to M\otimes \Lambda^{p-2}\gfg$ by induction on $p
\geq 2$.  In the case $p=2$, $\pa\pa = 0$ is equivalent to the condition
(\ref{A5action}). 
Assume $\pa\pa = 0: M\otimes \Lambda^p\gfg \to
M\otimes \Lambda^{p-2}\gfg$ for $p\geq 2$. 
Then, using (\ref{A5partial}) and (\ref{A5comm})
for $m \in M$, $\xi \in \Lambda^p\gfg$ and $Y \in\gfg$, we compute
\begin{eqnarray*}
&& \pa\pa(m\otimes\xi\wedge Y)
= \pa(\pa(m\otimes\xi)\wedge Y + (-1)^{p+1}\sigma(Y)(m\otimes\xi))\\
&=& \pa\pa(m\otimes\xi)\wedge Y + (-1)^p\sigma(Y)\pa(m\otimes\xi)
+ (-1)^{p+1}\pa(\sigma(Y)(m\otimes\xi))\\
&=& \pa\pa(m\otimes\xi)\wedge Y = 0
\end{eqnarray*}
by the inductive assumption. This proves the lemma. 
\end{proof}

The pair $(M, \sigma)$ is called a {\bf left $\gfg$-module} if
the map 
$\sigma$ satisfies the condition (\ref{A5action}). 
Then the $p$-th homology group of the chain complex 
$M\otimes\Lambda^*\gfg = \{M\otimes\Lambda^p\gfg, \pa\}_{p\geq0}$ is
denoted by 
$$
H_p(\gfg; M) = H_p(M\otimes\Lambda^*\gfg)
$$
and called the $p$-th homology group of the Lie algebra $\gfg$
with values in $M$. See, for example, \cite{CE}. \par
If we define $\overline{\sigma}: M\otimes \gfg \to M$ by 
$\overline{\sigma}(m\otimes X) = -Xm$ and the condition 
(\ref{A5action}) holds for $\sigma$, then the pair 
$(M, \overline{\sigma})$ is called a {\bf right $\gfg$-module}.
By the identification (\ref{A0wedge}) we have 
\begin{equation}
\Gamma_\sigma(m\otimes Y_1\wedge Y_2) = (\overline{\sigma}\otimes
1_\gfg)(m\otimes Y_1\wedge Y_2)
\label{A5gamma}
\end{equation}
for any $m \in M$ and $Y_1, Y_2 \in \gfg$.
\par

\subsection{Comodules}

Next let $(\gfg, \delta)$ be a Lie coalgebra, and $M$ a
$\mathbb{Q}$-linear space equipped with a $\mathbb{Q}$-linear map 
$\mu: M \to M\otimes\gfg$. We define a $\mathbb{Q}$-linear map 
$d^M = d: M\otimes\Lambda^p\gfg \to M\otimes\Lambda^{p+1}\gfg$, 
$p \geq 0$, by 
$$
d(m\otimes \xi) := \mu(m)\wedge\xi + (-1)^pm\otimes d\xi
$$
for $m \in M$ and $\xi \in\Lambda^p\gfg$. 
Here $d: \Lambda^p\gfg \to \Lambda^{p+1}\gfg$ 
is the operator introduced in \S\ref{LCA}. 
If $p=0$, then $d=\mu: M\to M\otimes\gfg$. From the definition 
and the formula (\ref{A2wedge}) follows
\begin{equation}
d(m\otimes\xi\wedge\eta) = d(m\otimes\xi)\wedge \eta +
(-1)^p(m\otimes\xi)\wedge(d\eta)
\label{A6wedge}
\end{equation}
for any $m \in M$, $\xi \in \Lambda^p\gfg$ and $\eta \in \Lambda^q\gfg$.

\begin{lem}\label{A6dd=0}
We have $d^Md^M=0: M\otimes\Lambda^*\gfg \to M\otimes\Lambda^*\gfg$, if
and only if the following diagram commutes
\begin{equation}
\begin{CD}
M @>{\mu}>> M\otimes\gfg\\
@V{\mu}VV @V{1_M\otimes\delta}VV\\
M\otimes\gfg @>{(1_M\otimes(1-T))(\mu\otimes
1_\gfg)}>> M\otimes\gfg\otimes\gfg
\end{CD}
\label{A6coaction}
\end{equation}
\end{lem}

\begin{proof} By (\ref{A0wedge}) we have
$$
d^M = (1_M\otimes(1-T))(\mu\otimes 1_\gfg) - 1_M\otimes\delta. 
$$
Here it should be remarked $d = -\delta: \gfg \to \gfg\otimes\gfg$. 
Hence the commutativity of the diagram (\ref{A6coaction}) is 
equivalent to $d^Md^M = 0$ on $M = M\otimes\Lambda^0\gfg$. 
In particular, $d^Md^M = 0$ implies the commutativity of the diagram
(\ref{A6coaction}). \par
Assume the diagram (\ref{A6coaction}) commutes. 
We prove $dd = 0: 
M\otimes \Lambda^p\gfg \to M\otimes \Lambda^{p+2}\gfg$ by induction on $p
\geq 0$.  In the case $p=0$, $dd = 0$ is equivalent 
to the commutativity of the diagram (\ref{A6coaction}). 
Assume $dd = 0:
M\otimes\Lambda^p\gfg \to M\otimes\Lambda^{p+2}\gfg$ for $p
\geq 0$.  Then, for $m\in M$, $\xi \in \Lambda^p\gfg$ and $Y \in\gfg$, 
we have 
$dd(m\otimes \xi\wedge Y) = d(d(m\otimes\xi)\wedge Y +
(-1)^pm\otimes\xi\wedge dY)  = dd(m\otimes \xi)\wedge Y +
(-1)^{p+1}d(m\otimes\xi)\wedge dY + (-1)^pd(m\otimes\xi)\wedge dY
+m\otimes\xi\wedge ddY = dd(m\otimes\xi)\wedge Y =
0$  by the inductive assumption. This proves the lemma.  
\end{proof}

The pair $(M, \mu)$ is called a {\bf right $\gfg$-comodule} if the
diagram (\ref{A6coaction}) commutes. 
Then the $p$-th cohomology group of the cochain complex 
$M\otimes\Lambda^*\gfg = \{M\otimes\Lambda^p\gfg, d\}_{p\geq0}$ 
is denoted by 
$$
H^p(\gfg; M) = H^p(M\otimes\Lambda^*\gfg)
$$
and called the $p$-th cohomology group of the Lie coalgebra $\gfg$
with values in $M$. 
In view of the formula (\ref{A6wedge}), $H^*(\gfg;M)$ is a graded 
right $H^*(\gfg)$-module.

\subsection{Involutive bimodules}\label{BM}

Let $\gfg$ be a Lie bialgebra, 
$(M, \overline{\sigma})$ a right $\gfg$-module,
and $(M, \mu)$ a right $\gfg$-comodule
with the same underlying vector space $M$. 
As in \S\ref{LBA}, 
we look at the operator $d^M\pa^M + \pa^Md^M: M\otimes\Lambda^p\gfg \to 
M\otimes\Lambda^p\gfg$ for $p \geq 0$. 
In \cite{KK4} Kuno and the author introduced {\bf the compatibility condition}
\begin{equation}
\forall m \in M, \forall Y \in \gfg, \quad
\sigma(Y)(dm) - d(Ym) = -\Gamma_\sigma(m\otimes dY),
\label{A6compatible}
\end{equation}
(or equivalently
\begin{equation}
\forall m \in M, \forall Y \in \gfg, \quad
\sigma(Y)(\mu (m)) - \mu(Ym) -(\overline{\sigma}\otimes
1_\gfg)(1_M\otimes\delta)(m\otimes Y) = 0,
\label{A6compatible+}
\end{equation}
) and {\bf the involutivity}
\begin{equation}
\overline{\sigma}\mu = 0: M \to M.
\label{A6involutive}
\end{equation}

\begin{lem}\label{A7axiom}
Let $\gfg$ be an involutive Lie bialgebra. Then 
we have $d^M\pa^M + \pa^Md^M = 0: M\otimes\Lambda^p\gfg \to 
M\otimes\Lambda^p\gfg$ for $p=0$ and $1$, if and only if
$\overline{\sigma}$ and $\mu$ satisfy {the compatibility condition} 
and {the involutivity}. 
\end{lem}

\begin{proof}
From the definition, the involutivity is equivalent to 
$d\pa+\pa d = 0$ for $p=0$. Assume the involutivity. 
Then, for $m \in \gfg$ and $Y \in \gfg$, we have 
$(d\pa+\pa d)(m\otimes Y) 
= -d(Ym) +\pa((dm)\wedge Y + m\otimes dY)
= -d(Ym) + (\pa dm)\wedge Y + \sigma(Y)(dm)
+\Gamma(m\otimes dY) + m\otimes\pa dY
= -d(Ym) + \sigma(Y)(dm)+\Gamma(m\otimes dY)$.
Hence $d\pa+\pa d = 0$ for $p=1$ is equivalent 
to the compatibility condition. This proves the lemma.
\end{proof}

For a Lie bialgebra $\gfg$, $M$ is called a {\bf 
right $\gfg$-bimodule} if the compatibility condition holds. A right
$\gfg$-bimodule $M$ is called  {\bf involutive}, if it satisfies the
involutivity.
\par

\begin{prop}\label{A7dd+dd}
If $\gfg$ is a Lie bialgebra, and $M$ a right
$\gfg$-bimodule, then we have 
\begin{eqnarray*}
&&(d\pa+\pa d)(m\otimes X_1\wedge\cdots\wedge X_p) \\
&=& (\pa dm)\otimes X_1\wedge\cdots\wedge X_p
 + m\otimes \sum^p_{i=1} X_1\wedge X_{i-1}\wedge (\pa dX_i)\wedge
X_{i+1}\wedge\cdots\wedge X_p
\end{eqnarray*}
for $m \in M$ and $X_i \in \gfg$. 
\end{prop}

\begin{proof}
It is clear for $p=0$. Assume it holds for $p \geq 0$.
Denote $\xi = X_1\wedge\cdots\wedge X_p$ and $Y = X_{p+1}$. 
We have 
$\sigma(Y)d(m\otimes\xi) - (dm)\wedge\sigma(Y)\xi -(Ym)\otimes d\xi
= (\sigma(Y)dm)\wedge\xi + m\otimes\sigma(Y)d\xi$.
So, by (\ref{A5partial}), (\ref{A5Gpartial}) and (\ref{A6wedge}), 
we compute
\begin{eqnarray*}
&&(d\pa+\pa d)(m\otimes\xi\wedge Y)\\
&=& (d\pa+\pa d)(m\otimes\xi)\wedge Y\\
&& +(-1)^pm\otimes\left(-(\pa\xi)\wedge dY - d\sigma(Y)(\xi)
+ \sigma(Y)d\xi + \pa(\xi\wedge dY)\right)\\
&& +(-1)^p\left(-(dYm) + \sigma(Y)dm + \Gamma(m\otimes dY)\right)\wedge\xi
\end{eqnarray*}
Hence, by Lemma \ref{A3partial} and (\ref{A6compatible}), we obtain
$$
(d\pa+\pa d)(m\otimes\xi\wedge Y) =
(d\pa+\pa d)(m\otimes\xi)\wedge Y 
+ m\otimes\xi\wedge\pa dY.
$$
This proceeds the induction.
\end{proof}

\begin{cor}\label{A7dd=0}
Let $\gfg$ be an involutive Lie bialgebra, and $M$ a right
$\gfg$-bimodule. Then we have $d^M\pa^M+\pa^Md^M = 0: 
M\otimes\Lambda^p\gfg \to M\otimes\Lambda^p\gfg$ for any 
$p \geq 0$, if and only if $M$ is involutive.
\end{cor}
This completes the proof of Proposition \ref{0bimod}.
\par
If $\gfg$ is an involutive Lie bialgebra and $M$ an involutive
right $\gfg$-bimodule, then the operator $d^M$ induces the cobounday
operator 
$$
d= d(\delta, \mu): H_p(\gfg; M) \to H_{p+1}(\gfg; M), \quad
[u]\mapsto [d^Mu]
$$
on the homology group $H_*(\gfg; M)$. 
Hence one can define the cohomology of the homology 
$H^*(H_*(\gfg; M))$.

\section{Drinfel'd's deformation}

Let $\gfg$ be a Lie algebra equipped with a Lie cobracket $\delta: \gfg \to
\Lambda^2\gfg$. As was pointed out by Drinfel'd \cite{D}, 
the compatibility is equivalent to that $\delta$ is a $1$-cocycle 
of the Lie algebra $\gfg$ with values in $\Lambda^2\gfg$, 
and so one can deform the cobracket $\delta$ by a $1$-coboundary 
of $\gfg$ with values in $\Lambda^2\gfg$ satisfying 
some condition which assures the new cobracket the coJacobi identity. 
Here $\gfg$ acts on $\Lambda^2\gfg$ by the map $\sigma: \gfg\otimes
\Lambda^2\gfg \to \Lambda^2\gfg$. The subspace $\mathcal{N}(\gfg)
:= \Ker(\nabla: \Lambda^2\gfg \to \gfg)$ is a $\gfg$-submodule. 
The involutivity means $\delta(\gfg) \subset \mathcal{N}(\gfg)$. 
Hence we may regard the set of involutive Lie bialgebra structures on
the underlying  Lie algebra $\gfg$ as a subset of 
$Z^1(\gfg; \mathcal{N}(\gfg))$, the set of $1$-cocycles
of $\gfg$ with values in $\mathcal{N}(\gfg)$.
In particular, we can say two cobrackets $\delta$ and $\delta'$, which define intolutive Lie bialgebra structures on $\gfg$, are {\bf cohomologous}
to each other
if and only if $[\delta] = [\delta'] \in H^1(\gfg; \mathcal{N}(\gfg))$. 
Similar observations hold for a involutive bimodule structure 
on a $\gfg$-module $M$. 
\par
We introduced the coboundary operators 
$d(\delta)$ and $d(\delta, \mu)$ 
on the homology group $H_*(\gfg)$ and $H_*(\gfg; M)$
in the previous sections. 
In this section, we prove that these operators 
stay invariant under Drinfel'd's deformation. 

\subsection{Deformation of a cobracket}

Let $\gfg$ be a Lie algebra. 

\begin{lem}\label{A4coboundary}
If $\delta$ and $\delta' \in Z^1(\gfg; \mathcal{N}(\gfg))$ 
are involutive Lie bialgebra structures on $\gfg$, and 
cohomologous to each other, then 
the induced coboundary operators $d(\delta)$ and $d(\delta')$ on the
homology $H_*(\gfg)$ coincide with each other
$$
d(\delta)=d(\delta'): H_*(\gfg) \to H_{*+1}(\gfg).
$$ 
\end{lem}
\begin{proof}
For $A \in \Lambda^*\gfg$, we denote by $E_A: \Lambda^*\gfg \to 
\Lambda^*\gfg$ the multiplication by $A$, $u \mapsto A\wedge u$. 
If $A \in \Lambda^2\gfg$, then, by some straight-forward computation,
we have 
\begin{equation}
(\pa E_A-E_A\pa+E_{\nabla A})(X_1\wedge\cdots\wedge X_p) 
= \sum^p_{i=1}(-1)^i\sigma(X_i)(A)\wedge
X_1\wedge\overset{\hat i}\cdots\wedge X_p
\label{A4homotopy}\end{equation}
for any $X_i \in \gfg$.\par
We denote $d = d(\delta)$ and $d' = d(\delta')$. 
Suppose $\delta$ and $\delta'$ are cohomologous to each other. 
Then there exists some $A \in \mathcal{N}(\gfg)$ such that 
$(d - d')(X) = (\delta'-\delta)(X) = \sigma(X)(A)$ for any 
$X\in \gfg$. From (\ref{A4homotopy}) follows
$(d'-d)(X_1\wedge\cdots\wedge X_p) 
= \sum^p_{i=1}(-1)^i\sigma(X_i)(A)\wedge
X_1\wedge\overset{\hat i}\cdots\wedge X_p
= (\pa E_A-E_A\pa+E_{\nabla A})(X_1\wedge\cdots\wedge X_p)$.
Since $A \in \mathcal{N}(\gfg)$, we obtain
$d'-d = \pa E_A-E_A\pa: \Lambda^*\gfg \to \Lambda^{*+1}\gfg$.
This proves the lemma.
\end{proof}

As was pointed out by Drinfel'd \cite{D}, 
we have $H^1(\gfg; \mathcal{N}(\gfg)) = 0$ 
in the case $\gfg$ is a finite-dimensional semi-simple
Lie algebra.  
Hence, in this case, $d(\delta) = 0$ on $H_*(\gfg)$ 
for any involutive Lie bialgebra structure on $\gfg$. \par

Let $U$ be an automorphism of a topological Lie algebra $\gfg$, and 
$\delta \in Z^1(\gfg; \mathcal{N}(\gfg))$ an involutive 
Lie bialgebra structure on $\gfg$. 
Then the conjugate $U\delta := (U\otimes U)\delta U^{-1}$ 
is also an involutive Lie bialgebra strucuture on $\gfg$. 
\begin{lem}\label{A4Aexp}
Let $X \in \gfg$, and 
suppose $e^{\ad X} = \sum^\infty_{k=0}\frac{1}{k!}(\ad(X))^k
$ converges
as an automorphism  of the topological Lie algebra $\gfg$. 
Then we have $d(\delta) = d(e^{\ad X}\delta)$ on 
$H_*(\gfg)$. 
\end{lem} 
\begin{proof}
The Lie algebra $\gfg$ acts on $Z^1(\gfg;\mathcal{N}(\gfg))$ 
in an obvious way. We have 
$$
(Yc)(Z) := \sigma(Y)(c(Z)) - c([Y,Z]) = \sigma(Z)(c(Y))
$$
for any $c \in Z^1(\gfg;\mathcal{N}(\gfg))$ and $Y, Z \in \gfg$. 
Now we have 
\begin{equation}
(Y^kc)(Z) = \sigma(Z)\sigma(Y)^{k-1}(c(Y))
\label{A4power}
\end{equation}
for any $k \geq 1$. If $k=1$, (\ref{A4power}) was already shown. 
Assume (\ref{A4power}) holds for $k \geq 1$. 
Then $(Y^{k+1}c)(Z) = \sigma(Z)\sigma(Y)^{k-1}((Yc)(Y)) =
\sigma(Z)\sigma(Y)^{k}(c(Y))$. This proceeds the induction.\par
Hence we have 
\begin{eqnarray*}
&&(e^{\ad X}\delta - \delta)(Z) =
\sum^\infty_{k=1}\frac1{k!}(X^k\delta)(Z)  =
\sum^\infty_{k=1}\frac1{k!}\sigma(Z)\sigma(X)^{k-1}(\delta X) \\
&=&
\sigma(Z)\left(\sum^\infty_{k=1}\frac1{k!}\sigma(X)^{k-1}\right)
(\delta X).
\end{eqnarray*}
This means $e^{\ad X}\delta - \delta$ is the $1$-coboundary induced 
by $\left(\sum^\infty_{k=1}\frac1{k!}\sigma(X)^{k-1}\right)
(\delta X)$. The lemma follows from Lemma \ref{A4coboundary}.
\end{proof}

\subsection{Deformation of a cobracket and a comodule structure map}

A similar results to Lemma \ref{A4coboundary} holds for 
a deformation of cobrackets and comodules. 

\begin{lem}\label{A8coboundary}
Let $\gfg$ be a Lie algebra, $M$ a $\gfg$-module, 
$\delta$ and $\delta' \in Z^1(\gfg;\mathcal{N}(\gfg))$ 
involutive Lie bialgebra structures on $\gfg$, and 
let $\mu$ and $\mu': M\to M\otimes\gfg$ make $M$ 
an involutive right $(\gfg, \delta)$-bimodule and 
an involutive right $(\gfg, \delta')$-bimodule,
respectively. Suppose there exist $A \in \mathcal{N}(\gfg)$ 
and $B \in \Lambda^2\gfg$ such that
\begin{enumerate}
\item[\rm (i)] $\forall X \in \gfg$, $(\delta' -\delta)(X) =
\sigma(X)(A)$, 
\item[\rm (ii)] $\forall m \in M$, $(\mu' -\mu)(m) = \pa(m\otimes B)$, and
\item[\rm (iii)] $\forall X \in \gfg$, $\sigma(X)(A)=\sigma(X)(B)$. 
\end{enumerate}
Then we have
$$
d(\delta, \mu) = d(\delta', \mu'): H_*(\gfg; M) \to H_{*+1}(\gfg; M).
$$
\end{lem}

\begin{proof}
We define $E_B: M\otimes\Lambda^p\gfg \to M\otimes\Lambda^{p+2}\gfg$
by $E_B(m\otimes\xi) := m\otimes\xi\wedge B$ for $m \in M$ and $\xi \in
\Lambda^p\gfg$. By (\ref{A4homotopy}) and (\ref{A5Gpartial}), 
we have 
$$
(\pa E_B - E_B\pa)(m\otimes\xi)
= \pa(m\otimes B)\wedge\xi + m\otimes
\sum^p_{i=1}(-1)^i\sigma(X_i)(B)\wedge X_1\wedge\overset{\hat
i}\cdots\wedge X_p.
$$
Using the conditions (ii) (iii) and (\ref{A4homotopy}), we compute
$
(\pa E_B - E_B\pa)(m\otimes\xi) = (\mu'-\mu)(m)\wedge\xi
+ m\otimes(\pa E_A - E_A\pa)\xi = (d'-d)(m\otimes\xi).
$
Here we write simply $d = d(\delta, \mu)$ and $d' = d(\delta', \mu')$.
This proves the lemma.
\end{proof}

Let $(\gfg, \delta)$ be a topological involutive Lie bialgebra, 
$(M, \mu)$ a topological involutive right $\gfg$-bimodule, 
$U$ an automorphism of the topological Lie algebra $\gfg$, 
and $U^M$ an automorphism of the topological vector space $M$ 
compatible with $U$. We define $U\mu := (U^M\otimes U)\delta
(U^M)^{-1}$. Then $(M, U\mu)$ is an involutive right 
$(\gfg, U\delta)$-bimodule. 
\begin{lem}\label{A8Aexp}
Let $X \in \gfg$ and suppose 
$e^{\ad X} = \sum^\infty_{k=0}\frac{1}{k!}(\ad(X))^k$ and 
$e^{\sigma(X)} = \sum^\infty_{k=0}\frac{1}{k!}(\sigma(X))^k$
converge
as automorphisms of the topological Lie algebra $\gfg$ and 
the topological vector space $M$, 
respectively. Then we have $d(\delta, \mu) = d(e^{\ad X}\delta, 
e^{\sigma(X)}\mu)$ on $H_*(\gfg; M)$.
\end{lem}
\begin{proof}
We write $A = \left(\sum^\infty_{k=1}\frac1{k!}\sigma(X)^{k-1}\right)
(\delta X)$. As was shown in Lemma \ref{A4Aexp}, 
$(e^{\ad X}\delta -\delta)(Z) = \sigma(Z)(A)$ for any $Z \in \gfg$. 
From \ref{A6compatible} follows $(X\mu)(m) = \Gamma(m\otimes\delta X)$. 
Let $\Phi \in \Lambda^2\gfg$. If we define $\varphi: M \to M\otimes\gfg$ 
by $\varphi(m) := \Gamma(m\otimes\Phi)$, then we have $(X\varphi)(m) 
= \sigma(X)\varphi(m) - \varphi(Xm) 
= \sigma(X)\Gamma(m\otimes\Phi) - \Gamma(Xm\otimes\Phi)
= \Gamma(m\otimes\sigma(X)\Phi)$. Hence, by $A \in \mathcal{N}(\gfg)$,  
\begin{eqnarray*}
&& (e^{\sigma(X)}\mu - \mu)(m) 
= \sum^\infty_{k=1}\frac1{k!}(X^k\mu)(m)
= \sum^\infty_{k=1}\frac1{k!}\Gamma(m\otimes\sigma(X)^{k-1}\delta X) \\
&=& \Gamma(m\otimes A) = \pa(m\otimes A).
\end{eqnarray*}
Consequently the lemma follows from Lemma \ref{A8coboundary}.
\end{proof}

\section{Surface Topology}

We discuss some relations among these homological facts and surface topology, 
in particular, a tensorial description of the Turaev cobracket and Kontsevich's 
non-commutative symplectic geometry. 

\subsection{Symplectic derivations}

It is the Lie algebra of symplectic derivations of the completed 
tensor algebra of a symplectic vector space that plays a central role
throughout this section. 
Let $H$ be a symplectic $\mathbb{Q}$-vector space of dimension $2g$, 
$g \geq 1$, and $\widehat{T} = \widehat{T}(H) := \prod^\infty_{m=0}
H^{\otimes m}$ the completed tensor algebra over $H$. 
$\widehat{T}$ is filtered by the two-sided ideals
$\widehat{T}_p := \prod^\infty_{m=p}H^{\otimes m}$, $p \geq 1$, 
and constitutes a complete Hopf algebra 
whose coproduct $\Delta: \widehat{T} \to \widehat{T}
\widehat{\otimes}\widehat{T}$ is given by $\Delta(X) = X\widehat{\otimes}1
+ 1\widehat{\otimes}X$ for any $X \in H$. 
The symplectic form $\omega \in H^{\otimes 2}$ is given by 
$\omega = \sum^g_{i=1}A_iB_i - B_iA_i \in H^{\otimes 2}$ for any 
symplectic basis $\{A_i, B_i\}^g_{i=1}$ of $H$. 
We study the Lie algebra of continuous derivations on $\widehat{T}$ 
annihilating the form $\omega$, which we denote by ${\rm Der}_\omega
(\widehat{T}) = \mathfrak{a}^-_g$. 
We regard ${\rm Der}_\omega(\widehat{T})$ as a subspace of $H^*\otimes 
\widehat{T}$ by the restriction map to $H$. The symplectic vector space $H$ 
is naturally isomorphic to its dual $H^*$ 
by the map $X \in H \mapsto (Y \mapsto X\cdot Y) \in H^*$, 
so that we identify $H^*\otimes \widehat{T} = H\otimes \widehat{T} 
= \widehat{T}_1$. Then the image of ${\rm Der}_\omega(\widehat{T})$
in $\widehat{T}_1$ coincides with the cyclic invariants in $\widehat{T}_1
= \prod^\infty_{m=1}H^{\otimes m}$. In other words, 
we identify ${\rm Der}_\omega(\widehat{T})$ with $N(\widehat{T}_1)
\subset \widehat{T}_1$, where $N: \widehat{T} \to \widehat{T}$ is 
the {\it cyclic symmetrizer} or the {\it cyclicizer} defined by $N\vert_{H^{\otimes 0}} := 0$ 
and $N(X_1\cdots X_m) := \sum^m_{i=1}X_i\cdots X_mX_1\cdots X_{i-1}$ 
for $X_i \in H$. See \cite{KK1} for details. The subspace $N(H^{\otimes 2})$ 
is a Lie subalgebra naturally isomorphic to $\mathfrak{sp}_{2g}(\mathbb{Q})$. 
\par
Schedler \cite{Sch} constructed a cobracket on the necklace Lie algebra 
associated to a quiver. The Lie algebra $\mathfrak{a}^-_g$ can be regarded 
as such a Lie algebra. Schedler's cobracket for $\mathfrak{a}^-_g$, 
which we denote by $\delta^{\rm alg}: \mathfrak{a}^-_g \to \mathfrak{a}^-_g\widehat{\otimes}\mathfrak{a}^-_g$, 
is given by 
\begin{eqnarray*}
\delta^{\rm alg}(N(X_1X_2\cdots X_m)) 
&=& \sum_{i<j}(X_i\cdot X_j)
\{N(X_{i+1}\cdots X_{j-1})\widehat{\otimes}
N(X_{j+1}\cdots X_mX_1\cdots X_{i-1})\\
&& \quad\quad - 
N(X_{j+1}\cdots X_mX_1\cdots X_{i-1})
\widehat{\otimes}
N(X_{i+1}\cdots X_{j-1})\}
\end{eqnarray*}
for any $X_i \in H$ and $m \geq 1$. \par
The cyclic symmetry suggests us a close relation between symplectic 
derivations and fatgraphs, which was exhausted in Kontsevich's 
formal symplectic geometry \cite{Kon}. 
He studied a Lie subalgebra $a_g := \bigoplus^\infty_{m=2}N(H^{\otimes m})$
of $\mathfrak{a}^-_g$, which he called ``associative", and proved that 
the primitive part of the limit of the relative homology $\lim_{g\to\infty} H_k(a_g, 
\mathfrak{sp}_{2g}(\mathbb{Q}))$ is isomorphic to $\bigoplus_{s>0, 2-2g-s<0}
H^{4g-4+2s-k}(\mathbb{M}^s_g/\mathfrak{S}_s; \mathbb{Q})$. 
Here $\mathbb{M}^s_g$ is the moduli space of Riemann surfaces 
of genus $g$ with $s$ punctures, and the $s$-th symmetric group 
$\mathfrak{S}_s$ acts on it by permutation of punctures. \par
Schedler's cobracket $\delta^{\rm alg}$ does not preserve the subalgebra 
$a_g$, so that $d(\delta^{\rm alg})$ does not act on the homology group 
$H_k(a_g)$. On the other hand, Schedler's cobracket $\delta^{\rm alg}$ 
preserves the subalgebra $a^-_g := \bigoplus^\infty_{m=1}N(H^{\otimes m})$, 
whose degree completion is just the Lie algebra $\mathfrak{a}^-_g$. 
\begin{problem}\label{agminus}
Find a fatgraph interpretation of the primitive part of the limit of the relative homology $\lim_{g\to\infty} H_k(a^-_g, 
\mathfrak{sp}_{2g}(\mathbb{Q}))$. 
\end{problem} 
The difference between $a_g$ and $a^-_g$ is just $H$, 
the derivations of degree $-1$, which seem to correspond to tails in fatgraphs. 
The homology group $H_*({a}_g^-, 
\mathfrak{sp}_{2g}(\mathbb{Q}))$ seems to be
related to the moduli space of Riemann surfaces with boundary and marked
points studied in \cite{Co}. 
See \cite{Pen} for details on fatgraphs. 
The coboundary operator 
$d(\delta^{\rm alg})$ is defined on $H_*(a^-_g, 
\mathfrak{sp}_{2g}(\mathbb{Q}))$, since $\delta^{\rm alg}$ is 
$\mathfrak{sp}_{2g}(\mathbb{Q})$-invariant, and vanishes on $N(H^{\otimes 2}) 
= \mathfrak{sp}_{2g}(\mathbb{Q})$. 
\begin{problem} If Problem \ref{agminus} is solved in an affirmative way, 
then find a fatgraph interpretation of the coboundary operator 
$d(\delta^{\rm alg})$. 
\end{problem}
As will be explained in the next subsection, Schedler's cobracket is closely 
related to the Turaev cobracket. So the operator 
$d(\delta^{\rm alg})$ seem to be related to degeneration of Riemann surfaces. 

\subsection{Turaev cobracket}

In this section, for simplicity, 
we confine ourselves to a compact connected oriented 
surface with connected boundary. 
See Appendix for the definitions of the Goldman bracket, the Turaev cobracket 
and the operations $\sigma$ and $\mu$ stated below. 
We begin by recalling some results of Kuno and the author 
on a completion of the Goldman Lie algebra \cite{KK1} \cite{KK3}.
Let $g \geq 1$ be a positive integer. 
We denote by $\Sigma = \Sigma_{g,1}$ a compact connected oriented 
surface of genus $g$ with $1$ boundary component, 
and by $\hat\pi = \hat\pi(\Sigma) = [S^1, \Sigma]$ the homotopy set of 
free loops on the surface $\Sigma$. 
Goldman \cite{Go} defines a natural Lie algebra structure
on the $\mathbb{Q}$-free vector space $\mathbb{Q}\hat\pi$, 
which we call the Goldman Lie algebra. 
Choose a basepoint $*$ on the boundary $\pa\Sigma$, and 
consider the fundamental group $\pi:= \pi_1(\Sigma, *)$.
The group ring $\mathbb{Q}\pi$ admits a decreasing filtration 
given by the power of the augmentation ideal $I\pi$. 
Since $\pi$ is a free group of rank $2g$, the completion map 
$\mathbb{Q}\pi \to \widehat{\mathbb{Q}\pi} :=\varprojlim_{n\to\infty}
\mathbb{Q}\pi/(I\pi)^n$ is injective. We can consider a similar 
completion of the Goldman Lie algebra $\mathbb{Q}\hat\pi$ 
as follows. 
The forgetful map of basepoints 
$\vert\,\, \vert: \mathbb{Q}\pi \to \mathbb{Q}\hat\pi$ is surjective, 
since $\Sigma$ is connected. We define a filtration 
$\{\mathbb{Q}\hat\pi(n)\}_{n\geq 1}$ of $\mathbb{Q}\hat\pi$ 
by $\mathbb{Q}\hat\pi(n) := \vert \mathbb{Q}1 + (I\pi)^n\vert$, 
where $1 \in \pi$ is the constant loop. In \cite{KK3} it is proved 
that $[\mathbb{Q}\hat\pi(n), \mathbb{Q}\hat\pi(n')] \subset 
\mathbb{Q}\hat\pi(n+n'-2)$. Hence we can consider the completed 
Goldman Lie algebra $\widehat{\mathbb{Q}\hat\pi} = \widehat{\mathbb{Q}\hat\pi}(\Sigma)$ defined by $\widehat{\mathbb{Q}\hat\pi} :=  \varprojlim_{n\to\infty}
\mathbb{Q}\hat\pi/\mathbb{Q}\hat\pi(n)$. 
In \cite{KK1} Kuno and the author 
defined a natural operation $\sigma: \mathbb{Q}\hat\pi\otimes \mathbb{Q}\pi 
\to \mathbb{Q}\pi$ to introduce a natural nontrivial 
$\mathbb{Q}\hat\pi$-module structure on the group ring $\mathbb{Q}\pi$,
which the completed group ring $\widehat{\mathbb{Q}\pi}$ inherits 
as a nontrivial $\widehat{\mathbb{Q}\hat\pi}$-module structure \cite{KK3}. 
These Lie algebras act on the algebras by (continuous) derivations, 
respectively. 
\par
As is classically known, the group ring ${\mathbb{Q}\pi}$ 
is embedded into the completed tensor algebra $\widehat{T}$ 
over the first rational homology group 
$H := H_1(\Sigma; \mathbb{Q})$ of the surface $\Sigma$ as (complete)
Hopf algebras. Here we consider $H$ a symplectic $\mathbb{Q}$-vector 
space by the intersection number on the surface $\Sigma$. 
To study the embedding in detail, 
Massuyeau \cite{Mas} introduced 
the notion of a {symplectic expansion} of the fundamental group $\pi$. 
A map $\theta: \pi \to \widehat{T}$ is a {\bf symplectic expansion} if 
it satisfies the following four conditions.
\begin{enumerate}
\item We have $\theta(xy) = \theta(x)\theta(y)$ for any $x$ and $y \in \pi$ . 
\item For any $x \in \pi$ we have $\theta(x) \equiv 1 + [x] \pmod{\widehat{T}_2}$, 
where $[x] \in H \subset \widehat{T}$ is the homology class of $x$.
\item For any $x \in \pi$, $\theta(x)$ is group-like, namely, $\Delta\theta(x) 
= \theta(x)\widehat{\otimes}\theta(x)$. 
\item Let $\zeta \in \pi$ be the boundary loop in the negative direction, 
and $\omega \in H^{\otimes 2} \subset \widehat{T}$ the symplectic form. 
Then we have $\theta(\zeta) = e^\omega \in \widehat{T}$.
\end{enumerate}
Symplectic expansions do exist \cite{Ka2} \cite{Mas} \cite{Ku1}. 
A symplectic expansion $\theta$ induces an isomorphism $\theta:
\widehat{\mathbb{Q}\pi} \overset\cong\to \widehat{T}$ of complete 
Hopf algebras. For any two symplectic expansions $\theta$ and $\theta'$, 
there exists an element of $u \in {\rm Der}_\omega(\widehat{T}) 
= \mathfrak{a}_g^-$ such that $(u\widehat{\otimes}u)\Delta = \Delta u$, 
$u(H) \subset \widehat{T}_2$ and $\theta' = e^u\circ\theta: 
\pi \to \widehat{T}$. See \cite{KK1} for details.
\par
In \cite{KK1} and \cite{KK3}, Kuno and the author proved 
\begin{thm}\label{isom}
Any symplectic expansion $\theta: \pi \to \widehat{T}$ induces 
\begin{enumerate}
\item an isomorphism of Lie algebras
$$
-N\theta: \widehat{\mathbb{Q}\hat\pi} 
\overset\cong\longrightarrow N(\widehat{T}_1) = 
{\rm Der}_\omega(\widehat{T}) = \mathfrak{a}^-_g
$$
given by $-(N\theta)(\vert x\vert) := -N(\theta(x))$ for any $x \in \pi$, and 
\item a commutative diagram
$$
\begin{CD}
\widehat{\mathbb{Q}\hat\pi}\otimes \widehat{\mathbb{Q}\pi} 
@>>> \widehat{\mathbb{Q}\pi} \\
@V{-N\theta\otimes\theta}VV @V{\theta}VV\\
{\rm Der}_\omega(\widehat{T})\otimes\widehat{T} 
@>>> \widehat{T},
\end{CD}
$$
where the horizontal arrows mean the actions as derivations. 
\end{enumerate}
\end{thm}
\par
Let $\mathbb{Q}\hat\pi' = \mathbb{Q}\hat\pi'(\Sigma)$ be the quotient 
of $\mathbb{Q}\hat\pi$ by the linear span of the constant loop $1 \in \hat\pi$. 
Since $1$ is in the center of $\mathbb{Q}\hat\pi$, it has a natural Lie algebra 
structure. In \cite{T2} Turaev introduced a cobracket $\delta$ on the Lie algebra 
$\mathbb{Q}\hat\pi'$ and proved that the pair $(\mathbb{Q}\hat\pi', \delta)$ 
is a Lie bialgebra. Later Chas \cite{Chas} proved that it is involutive. 
Kuno and the author \cite{KK4} proved the completed Goldman Lie 
algebra $\widehat{\mathbb{Q}\hat\pi}$ inherits the Turaev cobracket, 
so we call it the {completed Goldman-Turaev Lie bialgebra}.  
Inspired by Turaev's $\mu$ in \cite{T1}, 
they \cite{KK4} introduced a natural nontrivial 
comodule structure map $\mu: \mathbb{Q}\pi \to \mathbb{Q}\pi\otimes 
\mathbb{Q}\hat\pi'$, and proved that $(\mathbb{Q}\pi, \mu)$ is an 
involutive $\mathbb{Q}\hat\pi'$-bimodule. 
The comodule structure map $\mu$ defines a complete involutive $\widehat{
\mathbb{Q}\hat\pi}$-bimodule structure on the completed group ring 
$\widehat{\mathbb{Q}\pi}$ \cite{KK4}. \par
Let $\theta: \pi \to \widehat{T}$ be a symplectic expansion. 
Then the Turaev cobracket $\delta$ and 
the isomorphisms in Theorem \ref{isom} defines a cobracket 
$\delta^\theta := ((-N\theta)\widehat{\otimes}(-N\theta))\circ\delta\circ(-N\theta): 
\mathfrak{a}^-_g \to \mathfrak{a}^-_g\widehat{\otimes}\mathfrak{a}^-_g$. 
Similarly the comodule structure map $\mu^\theta: 
(\theta\widehat{\otimes}(-N\theta))\circ\mu\circ\theta: \widehat{T} \to 
\widehat{T}\widehat{\otimes}\mathfrak{a}^-_g$ can be defined so that 
$(\widehat{T}, \mu^\theta)$ is an involutive $\mathfrak{a}^-_g$-bialgebra. 
\par
The grading on $\mathfrak{a}_g^-$ defines the Laurent expansion of the cobracket $\delta^\theta$
\begin{eqnarray*}
&&\delta^\theta(N(X_1X_2\cdots X_m)) = \sum^\infty_{p= -\infty}\delta^\theta_{(p)}
(N(X_1X_2\cdots X_m)), \\
&&\quad \delta^\theta_{(p)}
(N(X_1X_2\cdots X_m)) \in (\mathfrak{a}_g^-\widehat{\otimes}\mathfrak{a}_g^-)_{(m+p)} := \bigoplus_{k+l = m+p} N(H^{\otimes k}) \otimes N(H^{\otimes l})
\end{eqnarray*}
for $X_i \in H$. 
Massuyeau and Turaev \cite{MT2} and Kuno and the author \cite{KK4} 
independently proved 
\begin{thm}\label{laurent}
 For any symplectic expansion $\theta$ we have 
\begin{enumerate}
\item $\delta^\theta_{(p)} = 0$ for $p = 0, -1$, and $p \leq -3$.
\item $\delta^\theta_{(-2)}$ is the same as Schedler's cobracket \cite{Sch}, i.e., 
 $\delta^\theta_{(-2)} = \delta^{\rm alg}$. 
\end{enumerate}
\end{thm}
Theorem \ref{laurent} follows from some computation 
based on a tensorial description 
of the homotopy intersection form by Massuyeau and Turaev \cite{MT1}.
In the computation we introduce the Laurent expansion of the comodule 
structure map $\mu^\theta$ in a similar way. The principal term is 
$\mu^{\rm alg}: \widehat{T} \to \widehat{T}\widehat{\otimes}\mathfrak{a}^-_g$
defined by 
$$
\mu^{\rm alg}(X_1\cdots X_m) := 
\sum_{1\leq i<j \leq m}(X_i\cdot X_j)X_1\cdots X_{i-1}X_{j+1}\cdots X_m\widehat{\otimes}N(X_{i+1}\cdots X_{j-1})
$$
for $X_i \in H$. The pair 
$(\widehat{T}, \mu^{\rm alg})$ is a complete involutive 
$(\mathfrak{a}^-_g, \delta^{\rm alg})$-bimodule.
So we present the following problem.
\begin{problem}
Find a fatgraph interpretation of the limit of the relative 
twisted homology $\lim_{g\to\infty} H_k(a^-_g, 
\mathfrak{sp}_{2g}(\mathbb{Q}); \widehat{T})$ and the coboundary operator 
$d(\delta^{\rm alg}, \mu^{\rm alg})$ on it. 
\end{problem}
 \par
As for the first term $\delta^\theta_{(1)}$ of the Laurent expansion of 
$\delta^\theta$, the following holds. 
\begin{prop}[\cite{KK5}]\label{first}
There exist symplectic expansions $\theta$ and $\theta'$ such that 
$\delta^\theta_{(1)} = 0$ and $\delta^{\theta'}_{(1)} \neq 0$.
\end{prop}
In particular, 
$\delta^\theta$ and $\mu^\theta$ do depend on the choice of 
a symplectic expansion $\theta$, and the cobracket $\delta^\theta$
for some $\theta$ does not coincide with Schedler's cobracket $\delta^{\rm alg}$. 
But the cohomology classes of $\delta^\theta$ and $\mu^\theta$ 
do not depend on the choice of symplectic expansions 
from the following proposition. 
\begin{prop}
Let $\theta'$ be another symplectic expansion. Then we have 
\begin{eqnarray*}
&& d(\delta^{\theta'}) = d(\delta^{\theta}) \quad\mbox{on
$H_*(\mathfrak{a}_g^{-})$, and}\\
&& d(\delta^{\theta'}, \mu^{\theta'}) = d(\delta^{\theta}, \mu^{\theta})
\quad\mbox{on
$H_*(\mathfrak{a}_g^{-}; \widehat{T})$.}
\end{eqnarray*}
\end{prop}
\begin{proof} There exists an element of $u \in {\rm Der}_\omega(\widehat{T}) 
= \mathfrak{a}_g^-$ such that $(u\widehat{\otimes}u)\Delta = \Delta u$, 
$u(H) \subset \widehat{T}_2$ and $\theta' = e^u\circ\theta: 
\pi \to \widehat{T}$. 
From some straight-forward computation in \cite{KK1} Lemma 4.3.1, 
we have $Ne^u = e^{{\rm ad}u}N: \widehat{T} \to \mathfrak{a}^-_g$. 
Therefore $\delta^{\theta'} = (e^{{\rm ad}u}\widehat{\otimes}e^{{\rm ad}u})\delta^\theta e^{-{\rm ad}u} = 
e^{{\rm ad}u}\delta$ and $\mu^{\theta'} =
(e^u\widehat{\otimes}e^{{\rm ad}u})\mu^\theta e^{-u} = e^{\sigma(u)}\mu^\theta$
in the sense of Lemma \ref{A8Aexp}. 
In view of Lemmas \ref{A4Aexp} and \ref{A8Aexp}, 
this shows the proposition. 
\end{proof}
This proposition makes us to present the following problems. 
\begin{problem}\label{B1cohp}
Determine whether $\delta^\theta$ and $\mu^\theta$ are cohomologous to Schedler's $\delta^{\rm alg}$ and $\mu^{\rm alg}$, respectively, or not. 
\end{problem}
\begin{problem}
If the answer to Problem \ref{B1cohp} is affirmative, 
determine 
whether there exists a symplectic expansion $\theta$ such that $\delta^\theta$ and $\mu^\theta$ coincide with Schedler's $\delta^{\rm alg}$ and $\mu^{\rm alg}$, respectively, or not. 
\end{problem}

\appendix

\section{Operations of loops on a surface}

In the appendix we briefly review some operations of loops on a surface 
introduced in \cite{Go} \cite{T2} \cite{KK1} and \cite{KK4}. 

\subsection{Goldman bracket}

Let $S$ be an oriented surface. We denote by $\hat\pi(S)$
the homotopy set of free loops on the surface $S$. 
For any $p \in S$ we denote by $\vert\,\,\vert: \pi_1(S, p) \to \hat\pi(S)$ 
the forgetful map of the basepoint $p$.  
Let $\alpha$ and $\beta$ be elements of $\hat\pi(S)$.  
We choose their representatives in general position, and 
denote them by the same symbols. 
Then the set of intersection points $\alpha\cap\beta$ is finite, 
and $\alpha$ and $\beta$ intersect transversely at each point 
in $\alpha\cap\beta$. The Goldman bracket is defined to be 
the formal sum
$$
[\alpha, \beta] := \sum_{p \in \alpha\cap\beta} \varepsilon_p(\alpha, \beta)
\vert \alpha_p\beta_p\vert
$$
in $\mathbb{Z}\hat\pi(S)$, the $\mathbb{Z}$-free module over the set 
$\hat\pi(S) = [S^1, S]$. Here $\varepsilon_p(\alpha, \beta) \in \{\pm1\}$ 
is the local intersection number at $p$, and $\alpha_p$ (resp.\ $\beta_p$)
$\in \pi_1(S, p)$ is the based loop along $\alpha$ (resp.\ $\beta$) 
with basepoint $p$. Goldman \cite{Go} proved that the bracket is 
well-defined, namely, homotopy invariant, and that the pair 
$(\mathbb{Z}\hat\pi(S), [\,,\,])$ is a Lie algebra, which we call the 
Goldman Lie algebra of the surface $S$. \par
Assume that the boundary $\pa S$ is non-empty, and let $*$ 
be a point on the boundary $\pa S$. We denote by $\Pi S(p_0, p_1)$ 
the homotopy set of paths on $S$ from $p_0$ to $p_1 \in S$. 
Choose representatives of $\alpha \in \hat\pi(S)$ and $\gamma \in \pi_1(S, *)$ 
in general position. The formal sum
$$
\sigma(\alpha)(\gamma) := \sum_{p \in \alpha\cap\gamma}
\epsilon_p(\alpha, \gamma)\gamma_{*p}\alpha_p\gamma_{p*} \in 
\mathbb{Z}\pi_1(S, *)
$$
is well-defined, namely, homotopy invariant \cite{KK1}. 
Here $\gamma_{*p} \in \Pi S(*,p)$ (resp.\ $\gamma_{p*} \in \Pi S(p, *)$)
is (the homotopy class of) 
the restriction of $\gamma$ to the segment from $*$ to $p$ 
(resp.\ from $p$ to $*$). Moreover $\sigma$ defines a Lie algebra 
homomorphism $\sigma: \mathbb{Z}\hat\pi(S) \to
 {\rm Der}(\mathbb{Z}\pi_1(S, *))$ \cite{KK1}. 
 If $*_0$ and $*_1$ are two distinct points on $\pa S$, then 
 $\mathbb{Z}\Pi S(*_0, *_1)$, 
 the $\mathbb{Z}$-free module over the set $\Pi S(*_0, *_1)$, 
 has a similar $\mathbb{Z}\hat\pi(S)$-module structure \cite{KK3}. 
\par 
\subsection{Turaev cobracket}
Let $S$ be a connected oriented surface. The constant loop $1 \in \hat\pi(S)$ 
on the surface is in the center of the Goldman Lie algebra 
$\mathbb{Z}\hat\pi(S)$, so that the quotient $\mathbb{Z}\hat\pi'(S) := 
\mathbb{Z}\hat\pi(S)/\mathbb{Z}1$ has a natural Lie algebra structure. 
We denote by $\vert\,\,\vert': \mathbb{Z}\pi_1(S, p) \to  \mathbb{Z}\hat\pi'(S)$ 
the composite of the forgetful map of the base point $p \in S$ and 
the quotient map $\mathbb{Z}\hat\pi(S)\to\mathbb{Z}\hat\pi'(S)$. 
Choose a representative of $\alpha \in \hat\pi(S)$ in general position, and 
denote it by the same symbol. Then the set $D_\alpha := \{(t_1, t_2) 
\in S^1\times S^1; \, t_1\neq t_2, \, \alpha(t_1)= \alpha(t_2)\}$ is finite 
and $\alpha$ intersects itself transversely at each $\alpha(t_1)= \alpha(t_2)$. 
The Turaev cobracket is defined to be the formal sum
$$
\delta(\alpha):=\sum_{(t_1,t_2)\in D_\alpha}
\varepsilon(\dot{\alpha}(t_1),\dot{\alpha}(t_2)) \vert\alpha_{t_1t_2}\vert' \otimes
|\alpha_{t_2t_1}|^{\prime}
$$
in $\mathbb{Z}\hat{\pi}'(S)
\otimes \mathbb{Z}\hat{\pi}'(S)$. Here $\varepsilon(\dot{\alpha}(t_1),\dot{\alpha}(t_2))
\in \{\pm1\}$ 
is the local intersection number of the velocity vectors $\dot{\alpha}(t_1)$ 
and $\dot{\alpha}(t_2) \in T_{\alpha(t_1)}S$, and $\alpha_{t_1t_2}$ 
(resp.\ $\alpha_{t_2t_1}$) $\in \pi_1(S, \alpha(t_1))$ is (the homotopy class) 
of the restriction of $\alpha$ to the interval $[t_1,t_2]$ (resp.\ $[t_2,t_1]$). 
Turaev \cite{T2} proved that the cobracket $\delta$ is well-defined, namely, 
homotopy invariant, and that the pair $(\mathbb{Z}\hat{\pi}'(S), \delta)$ 
is a Lie bialgebra. Later Chas \cite{Chas} proved that it satisfies the 
involutivity. \par
Assume that the boundary $\pa S$ is non-empty, and let $*$ 
be a point on the boundary $\pa S$. 
The homomorphism $\sigma$ stated above factors through 
the quotient $\mathbb{Z}\hat\pi'(S)$. 
Choose a representative of $\gamma \in \pi_1(S, *)$ such that it is
a smooth immersion whose singularities are at most ordinary double points, 
the image of the interior $]0,1[$ is included in the interior of $S$, and 
the velocity vectors at the endpoints $0$ and $1$ are linearly independent on 
the tangent space $T_*S$. We denote it by the same symbol $\gamma$. 
Then the set $\Gamma_\gamma$ of self-intersection points of $\gamma$ 
except $*$ is finite. 
For $p \in \Gamma_p$, we denote $\gamma^{-1}(p) = \{t^p_1, t^p_2\}$ so that 
$t^p_1 < t^p_2$. 
Inspired by Turaev \cite{T1}, Kuno and the author \cite{KK4} introduced 
the formal sum
$$
\mu(\gamma) := 
\begin{cases}
-\sum_{p \in \Gamma_\gamma}\varepsilon(\dot{\gamma}(t^p_1), 
\dot{\gamma}(t^p_2))(\gamma_{0t^p_1}\gamma_{t^p_21})\otimes
\vert\gamma_{t^p_1t^p_2}\vert', & \text{if $\varepsilon(\dot{\gamma}(0), 
\dot{\gamma}(1)) = +1$},\\
1\otimes\vert\gamma\vert' -\sum_{p \in \Gamma_\gamma}\varepsilon(\dot{\gamma}(t^p_1), 
\dot{\gamma}(t^p_2))(\gamma_{0t^p_1}\gamma_{t^p_21})\otimes
\vert\gamma_{t^p_1t^p_2}\vert', & \text{if $\varepsilon(\dot{\gamma}(0), 
\dot{\gamma}(1)) = -1$}
\end{cases}
$$
in $\mathbb{Z}\pi_1(S, *)\otimes\mathbb{Z}\hat\pi'(S)$. 
Here $\gamma_{\tau_0\tau_1} \in \Pi S(\gamma(\tau_0), \gamma(\tau_1))$ 
is (the homotopy class of) the restriction of $\gamma$ to the interval $[\tau_0, 
\tau_1] \subset [0,1]$ for $0 \leq \tau_0 \leq \tau_1\leq 1$. 
They proved that the map $\mu$ is well-defined, namely, homotopy invariant, 
and that the pair $(\mathbb{Z}\pi_1(S, *), \mu)$ is an involutive 
$\mathbb{Z}\hat\pi'(S)$-bimodule \cite{KK4}. 
 If $*_0$ and $*_1$ are two distinct points on $\pa S$, then 
 $\mathbb{Z}\Pi S(*_0, *_1)$ 
 has a similar involutive $\mathbb{Z}\hat\pi'(S)$-bimodule structure \cite{KK4}. 
\par

\par
\bigskip

\noindent \textsc{Nariya Kawazumi\\
Department of Mathematical Sciences,\\
University of Tokyo,\\
3-8-1 Komaba Meguro-ku Tokyo 153-8914 JAPAN}\\
\noindent \texttt{E-mail address: kawazumi@ms.u-tokyo.ac.jp}

\end{document}